\documentclass[11pt]{amsart}
\usepackage{amsmath,amssymb,epsfig,color}%slashbox}
\usepackage{graphicx}
\usepackage{multirow}
\usepackage{hhline}
\usepackage{url}
\usepackage{color}
\graphicspath{ {images/} }
\newtheorem{theorem}{Theorem}[section]

\newtheorem{corollary}[theorem]{Corollary}

\newtheorem{lemma}[theorem]{Lemma}
\newtheorem{proposition}[theorem]{Proposition}
\newtheorem{remark}[theorem]{Remark}
\newtheorem{example}[theorem]{Example}

\newcommand{\vanish}[1]{}\parskip=12pt

\def\p{\prime}

\def\R{\mathbb{R}}

\def\A{\mathcal{A}}

\def\C{\mathcal{C}}
\def\D{\mathcal{D}}

\def\N{\mathcal{N}}
\def\T{\mathcal{T}}
\def\U{\mathcal{U}}
\def\V{\mathcal{V}}
\def\F{\mathcal{F}}
\def\tl{\triangleleft}
\def\tr{\triangleright}
\numberwithin{equation}{section}

\begin{document}

\title[The HOMFLY Polynomial of Closed Braids]
{The HOMFLY Polynomial of Links in Closed Braid Form}
\author{Pengyu LIU, Yuanan DIAO and G\'abor Hetyei}
\address{Department of Mathematics and Statistics, UNC Charlotte,
    Charlotte, NC 28223}
\email{}
%\thanks{$^\dag$supported by NSF grant DMS-0712958.
%$^\ast$supported by NSA grant \# H98230-07-1-0073} \dedicatory{}

\begin{abstract}
It is well known that any link can be represented by the closure of a braid. The minimum 
number of strings needed in a braid whose closure represents a given link is called the 
braid index of the link and the well known Morton-Frank-Williams inequality reveals a close relationship between the HOMFLY polynomial of a link and its braid index. In the case that a link is already presented in a closed braid form, Jaeger derived a special formulation of the HOMFLY polynomial. In this paper, we prove a   variant of Jaeger's result as well as a dual version of it. Unlike Jaeger's original reasoning, which relies on representation theory, our proof uses only
elementary geometric and combinatorial observations. Using our variant
and its dual version, we provide a direct and elementary proof of the
fact that the braid index of a link that has an $n$-string closed braid
diagram that is also reduced and alternating, is exactly $n$. Until know
this fact was only known as a consequence of a result due to Murasugi on
fibered links that are star products of elementary torus links and of
the fact that alternating braids are fibered. 
\end{abstract}

\maketitle
\section{Introduction}

\medskip
It is well known that any link can be represented by the closure of a braid. The minimum 
number of strands needed in a braid whose closure represents a given link is called the 
{\em braid index} of the link. Defined as the extreme value of a quantity over an infinite family of links that are topologically equivalent like other link invariants (such as the minimum crossing number),  the braid index of a link is hard to
compute~\cite{A} in general. In the case of the minimum crossing number, there was a long standing conjecture which states that for a reduced alternating link diagram, the number of crossings in the diagram is equal to the minimum crossing number of the link. It is known that the span of the Jones polynomial \cite{Jo} of a link gives a lower bound on the crossing number of the link. In the case of a reduced alternating link diagram, it was shown that the span of the Jones polynomial of a link equals the number of crossings in the diagram, which leads to the proof of the conjecture \cite{K,M1,T}. In the case of braid index, there is a similar inequality relating the braid index of a link to the $a$-span of its HOMFLY polynomial (which is a polynomial of two variables $z$ and $a$ to be defined in the next section). S.\ Yamada proved that any link diagram of a given link $L$ with $k$ {\em Seifert circles} can be realized as the closure of a braid on $k$ strands, which implies that the braid index of an oriented link $L$ equals the minimum number of Seifert circles of all link
diagrams of $L$~\cite{Ya}. In ~\cite{Mo}, H.\ Morton showed that the
number of Seifert circles of a link $L$, hence the braid index of $L$
(in light of Yamada's result),  is bounded from below by $a$-span$/2+1$
(which is called the Morton-Frank-Williams inequality, or MFW inequality
for short). In analogy to the crossing number conjecture for a reduced
alternating link diagram, K.~Murasugi conjectured that the number of
Seifert circles, hence the braid index, in such a diagram equals $a$-span$/2+1$ (the Murasugi Conjecture) \cite{Mu}. Although this conjecture turned out to be false in general~\cite{MP} (for example the knot $5_2$ has 4 Seifert circles, but the $a$-span of its HOMFLY polynomial is 4 so $a$-span$/2+1=3$ and the braid index of $5_2$ is also 3), researchers had shown that the MFW inequality is sharp for many classes of links (including some non-alternating ones) hence the $a$-span of the HOMFLY polynomial for these links can be used to determine their braid index. Examples include the closed positive braids with a full twist (in particular the torus links)~\cite{FW}, the 2-bridge links and fibered alternating links~\cite{Mu}, and a new class of links discussed in a more recent paper by S.~Y.~Lee and M.~Seo~\cite{Lee}. For more readings on related topics, interested readers can refer to J.S.~Birman and W.W.~Menasco~\cite{Bir}, P.R.~Cromwell~\cite{Crom}, E.A.~Elrifai~\cite{El},  H.~Morton, H.~B.~Short~\cite{MS}, T.~Nakamura~\cite{Na1} and A.~Stoimenow~\cite{Sto}.

Motivated by this question, in this paper the authors seek a special and explicit formulation of the HOMFLY polynomial for certain classes of links where the explicit forms of the HOMFLY polynomial would allow us to analyze and derive the $a$-spans of the HOMFLY polynomials of these links. Our main result expresses the HOMFLY polynomial of a link presented in a closed braid diagram form in two explicit formulas. We show that one of our formulas of the HOMFLY polynomial is equivalent to the expansion derived by F.\ Jaeger ~\cite{Ja}. However, our approach is combinatorial in nature and the proof of the formula is shorter than the proof in ~\cite{Ja}. We show that the Morton-Frank-Williams inequality is an immediate consequence of these two HOMFLY polynomial formulas. As a more significant application of our result, we use it to show that if a link has an $n$-string closed braid diagram that is also reduced and alternating, then the braid index of the link is exactly $n$. The proof of this is direct and short. Our approach is very different from the proof of Murasugi's more general results ~\cite{Mu} on oriented alternating fibered links, which is inductive in nature. 

This paper is structured as follows. In Section~\ref{s2}, we introduce
basic definitions and terminology about link diagrams, braids and the HOMFLY
polynomial. We introduce two special classes of resolving trees for
closed braids, in which every vertex is a closed braid and the leaf vertices are closed braids that
represent trivial links. We call these resolving
trees descending trees and ascending trees respectively. In
Section~\ref{s3}, we state and prove our main result, namely two formulas expressing the HOMFLY
polynomial of a closed braid as a total 
contribution of all leaf vertices in our trees and show that the Morton-Frank-Williams inequality~\cite{FW,Mo} is a direct consequence of these two formulas. In Section~\ref{s4}, we show that the $a$-span of the HOMFLY
polynomial of a reduced alternating braid on $n$ strands is exactly
$2n-2$ and an application of the the Morton-Frank-Williams inequality shows that the braid index of
a reduced alternating braid equals the number of strands in the braid. As another application of our main result from Section~\ref{s3}, we also give a short proof that the leading coefficient of the Alexander polynomial of such a closed braid equals $\pm 1$.
Finally, in Section~\ref{acp}, 
we show that one of our formulas is equivalent to the expansion of the HOMFLY polynomial derived by 
F.\ Jaeger~\cite{Ja} based on the concept of admissible circuit partitions. 

%\bigskip
\section{Basic concepts}\label{s2}

\subsection{Link diagrams and Reidemeister moves}

We assume that the reader has the basic knowledge about the definition of a link and refer a reader without such knowledge to a textbook such as \cite{A, B, Li}. For the convenience of the reader, however, we will review one important result that is needed in Section \ref{s3}. 

Figure \ref{R-moves} defines three moves one can make on a link diagram without changing its topology, and these are called Reidemeister moves of type I, II and III. In 1926, K.~Reidemeister~\cite{Re} proved that two link diagrams represent the same link if and only if one diagram can be changed to the other through a finite sequence of Reidemeister moves.

\begin{figure}[htb!]
\includegraphics[scale=0.6]{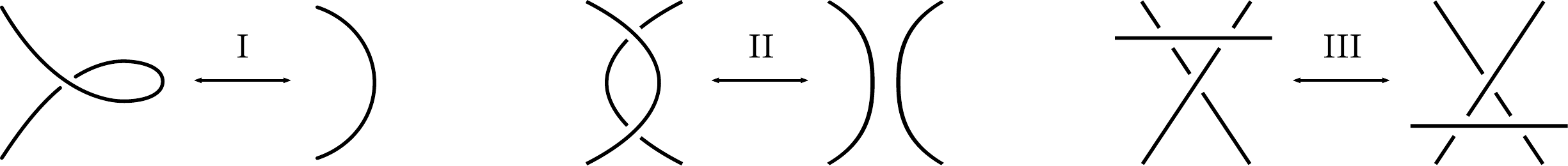}
\caption{Reidemeister moves of type I, II and III.}
\label{R-moves}
\end{figure}

For a given oriented link diagram $L$, we assign each crossing $+1$ or $-1$ according to its sign as defined in Figure~\ref{fig:cross}. The writhe of $L$, written as $w(L)$, is the sum of these $+1$'s and $-1$'s over all crossings of $L$. However, if we only sum the $+1$'s and $-1$'s over crossings between two different components $\C_1$ and $\C_2$ of the link, then this number is always even and half of it is called the {\em linking number} between $\C_1$ and $\C_2$. Using the Reidemeister moves, it is easy to show that the linking number is a link invariant. In particular, if $\C_1$ and $\C_2$ are separated by a topological plane (in which case $\C_1$ and $\C_2$ are said to be {\em splittable}), then the linking number between $\C_1$ and $\C_2$ is zero. We make the following remark for future reference. Its proof is trivial and is left to the reader.

\begin{remark}\label{R_remark}
{\em Reidemeister moves of type II and III do not change the writhe of a link diagram hence link diagrams related by a finite sequence of Reidemeister moves II and III have the same writhe. In particular, if a link diagram $L_2$ is obtained through $L_1$ by a finite sequence of moves that involve only deformation of segments of the link within planes that are parallel to the projection plane of the link diagram, then $w(L_2)=w(L_1)$ since such moves do not introduce Reidemeister moves of type I. Also, if $L$ is a link with splittable components $\C_1$, $\C_2$, \ldots, $\C_{\tau}$, then $w(L)=\sum_{1\le j\le \tau}w(\C_j)$.}
\end{remark}

A crossing in a link diagram is called {\em nugatory} if there is a simple closed curve such that the link diagram intersects the curve only at the crossing and the link diagram intersects both components of the complement of the curve. 
A link diagram is said to be {\em reduced} if no crossing in the diagram is nugatory. A link diagram is {\em alternating} if as we travel through the link diagram by any given orientation, the strands go through the crossings alternately between overpasses and underpasses. For example, the closure of the braid in Figure \ref{fig:bd} is an alternating link diagram. As we mentioned in the introduction, reduced alternating link diagrams are special since the number of crossings in a reduced alternating link diagram is the minimum crossing number of the link  \cite{K,M1,T}.

\subsection{Braids}

Consider $\R^2$ as the standard Euclidean $xy$-plane.  A {\em braid diagram} (or just a {\em braid}) on $n$ strands is a set  $\D\subset \R\times [0,1]$ consisting of $n$ curves called $strands$ of $\D$ such that the following four conditions are met. First, each strand is a monotonic curve in the $y$ direction. Second, every point of $\{1,2,\ldots,n\}\times\{0\}$ is a starting point of a unique strand and every point of $\{1,2,\ldots,n\}\times\{1\}$ is an end point of a unique strand. Third, every point of $\R\times I$ belongs to at most two strands. A point that belongs to two strands is called a {\em crossing}. At each crossing, one strand is identified as an {\em overpass} and the other is as an {\em underpass}~\cite{Ka}. Fourth, there is at most one crossing in $\R\times \{t\}$ for each $t\in [0,1]$. Note that the second condition gives the braid diagram a downward orientation, so the closure of a braid diagram is an oriented link diagram. 

Treated as topological objects, one can speak of topological equivalence
of braid diagrams, and such equivalence relations provide the
foundations for one to treat the braids as elements in the algebraic
objects called the {\em braid groups}. Not to deviate from our main
task, we will only point out that a braid group $B_n$ is a group with
$n-1$ generators $\sigma_1, \ldots, \sigma_{n-1}$ satisfying certain
relations. An element of $B_n$ is a 
word of these generators and each letter in the word corresponds to a
crossing in the braid. An example of a braid on 5 strands
and its counterpart in $B_5$ is given in Figure \ref{fig:braid}. Please
refer to ~\cite{Ka} for more details on braid groups. In this paper we
are only interested in closed braids as topological objects, not the
braids in the algebraic sense. For this reason, {\em we will not distinguish a braid and its closure}, that is, the word braid (or a symbol of it) can either represent the braid itself or its closure. The reader should rely on the context to determine its meaning (in many cases it really does not matter). 

We define the {\em label} of a strand by the $x$ coordinate of its starting point and we denote the corresponding mapping by $\ell$, that is, if a strand $s$ starts at $(m,1)$, then $\ell(s)=m$. On the
other hand, the mapping $p$ that takes the label of a strand to the $x$ coordinate of its ending point defines a permutation of the labels (namely the integers from $1$ to $n$). Denote this permutation by $p(\D)$ and write it as a product of disjoint cycles, we have $p(\D)=(s_{11}s_{12} \ldots s_{1k_1})(s_{21}s_{22}\ldots s_{2k_2})\ldots (s_{\tau 1}s_{\tau 2}\ldots s_{\tau k_\tau })$ where $s_{i1}$ is the label of the first strand in the $i$-th cycle, $s_{i2}$ is the label of the second strand in the $i$-th cycle, ..., and so on, $\tau $ is the number of cycles in the permutation. Furthermore, we can re-arrange the orders of the cycles and the numbers in each cycle so that $s_{i1}$ is the smallest integer in each cycle for each $i$, and $s_{i1}<s_{j1}$ if $i<j$. We call this special form of $p(\D)$ the {\em  standard form}. From now on, $p(\D)$ will always be expressed in its standard form. Note that the standard form of $p(\D)$ defines a total order among the strand labels, namely
\begin{displaymath}
s_{11}\tl s_{12}\tl \ldots \tl s_{1k_1}\tl s_{21}\tl s_{22}\tl \ldots \tl s_{2k_2}\tl \ldots \tl s_{\tau 1}\tl s_{\tau 2}\tl \ldots\tl s_{\tau k_\tau}
\end{displaymath}
We call this order the {\em return order} of the strands in the braid diagram $\D$. 

We call each $s_{i1}$ in $p(\D)$ the {\em pivot label} within its corresponding cycle and $(s_{i1},1)$ the {\em pivot point} of the cycle when $p(\D)$ is expressed in its standard form. Note that each cycle in $p(\D)$ corresponds to a connected component in the closed braid diagram and we can travel through $\D$ by traveling through each such component. We say that we travel through $\D$ {\em naturally} if we travel along the strands of $\D$ in its return order, starting from the pivot point at each component and follow the orientation of the braid $\D$.

\begin{figure}[htb!]
\includegraphics[scale=.4]{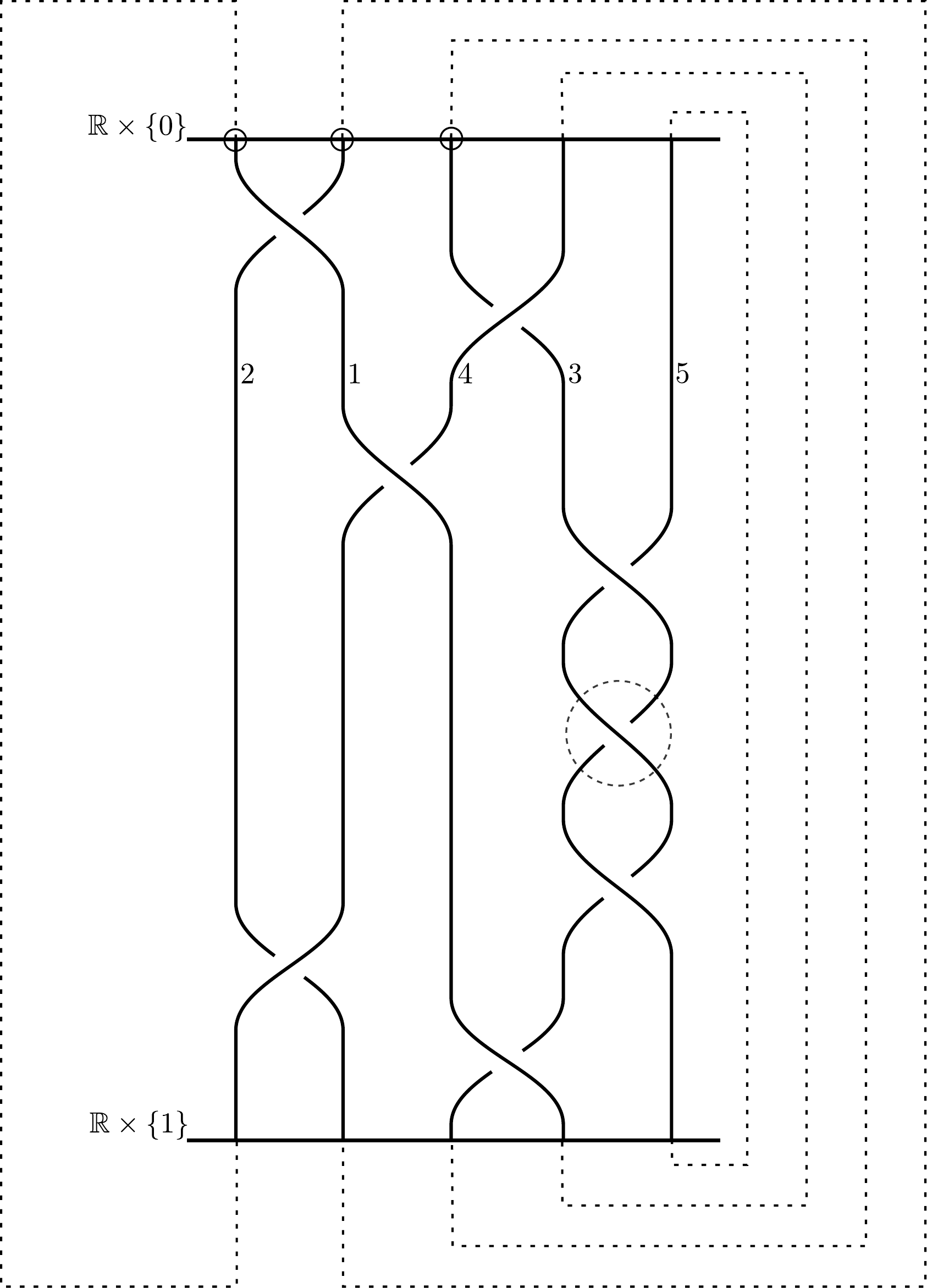}
\caption{An example of a braid $\D$ on 5 strands that has 8 crossings. }
\label{fig:braid}
\end{figure}

\begin{example} {\em Consider the braid diagram $\D$ shown in Figure~\ref{fig:braid} whose corresponding word in $B_5$ is $\sigma_1^{-1}\sigma_3\sigma_2^{-1}\sigma_4^{-3}\sigma_1\sigma_3^{-1}$. The dotted curves show how the braid is closed to form a link diagram. The circled points on top are pivot points of the three components. We have $p(\D)=(14)(2)(35)$ (in its standard form), so its return order is $1\tl 4\tl 2\tl 3\tl 5$. One can see that if we close the braid using disjoint curves starting and ending at the points $(i,1)$, $(i,0)$ for each $1\le i\le 5$ (shown in Figure \ref{fig:braid} with dotted curves), strands $1$ and $4$ are in the same component, strand $2$ is a component by itself and strands $3$ and $5$ are in the same component. The pivot labels of these components are $1$, $2$ and $3$ respectively.} 
\end{example}

Let $c$ be a crossing in a braid diagram $\D$, $O_c$ be the overpassing strand at $c$ and $U_c$ be the underpassing strand at $c$. We say that $c$ is {\em descending} if $\ell(O_c)\tl \ell(U_c)$ and {\em ascending} if $\ell(O_c)\tr \ell(U_c)$. If all crossings in $\D$ are descending (ascending), then we say that $\D$ is a descending (ascending) braid diagram. For example, all crossings except the circled one in Figure \ref{fig:braid} are descending crossings. Switching the circled crossing to a descending crossing will then results in a descending braid. 

\begin{remark}\label{descending_remark}
{\em Descending (ascending) braid diagrams have the easy and well known property that the closure of any such braid diagram is a trivial link, that is, a link topologically equivalent (ambient isotopic) to the disjoint union of several circles contained in the same plane. Take the braid given in Figure~\ref{fig:braid} as an example by flipping the circled crossing to make the braid descending. Using the order $1\tl 4\tl 2\tl 3\tl 5$, we can actually move strands 1, 4, 2, 3 and 5 into planes $z=5$, 4, 3, 2 and 1 respectively so that their projections would still give the descending braid, and that this move will not change the topology of the corresponding link (since there are no crossing changes in the process). It is then easy to see that each component corresponding to a cycle $(14)$, $(2)$ or $(35)$ is an unknot and these knots are splittable since they are separated by planes. }
\end{remark}

%\bigskip
\subsection{The HOMFLY polynomial} Let $L_+$, $L_-$, and $L_0$ be oriented link diagrams that coincide except at a small neighborhood of a crossing where the diagrams are presented as in Figure~\ref{fig:cross}: the crossing in $L_+$ ($L_-$) is positive (negative) and is assigned $+1$ ($-1$) in the calculation of the writhe of the link diagram. We say the crossing presented in $L_+$ is of a {\em positive} sign and the crossing presented in $L_-$ is of a {\em negative} sign. The following result appears in~\cite{Fr,Ja}.

\begin{proposition}\label{Ho}
There is a unique function that maps each oriented link diagram $L$ to a two-variable Laurent polynomial with integer coefficients $P(L,z,a)$ such that 
\begin{enumerate}
\item If $L_1$ and $L_2$ are ambient isotopic, then $P(L_1,z,a)=P(L_2,z,a)$.
\item $aP(L_+,z,a) - a^{-1}P(L_-,z,a) = zP(L_0,z,a)$. 
\item If L is an unknot, then $P(L,z,a)=1$. 
\end{enumerate}
\end{proposition}

\begin{figure}[htb!]
\includegraphics[scale=.6]{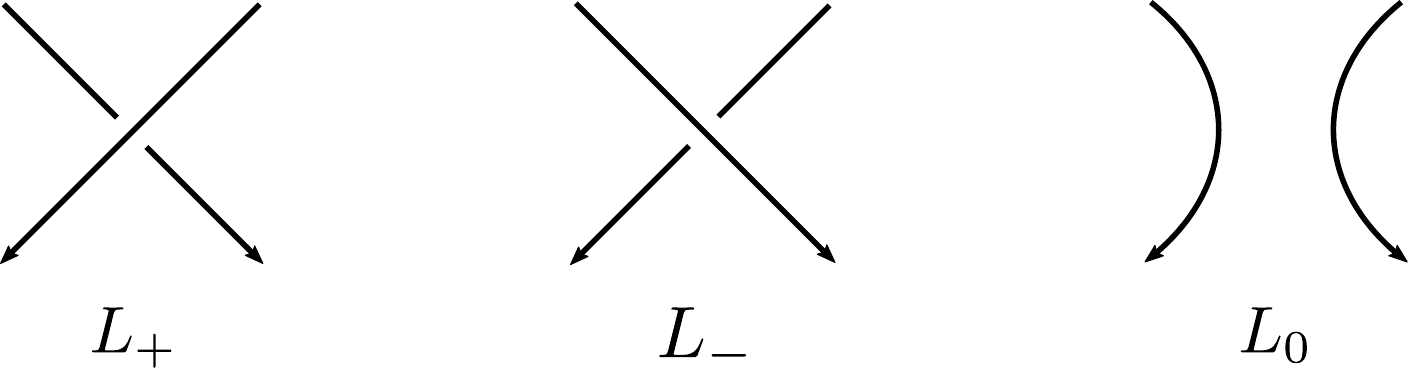}
\caption{The sign convention at a crossing of an oriented link and the splitting of the crossing.}
\label{fig:cross}
\end{figure}

The Laurent polynomial $P(L,z,a)$ is called the {\em HOMFLY polynomial}
of the oriented link $L$. The second condition in the proposition is
called the {\em skein relation} of the HOMFLY polynomial. With conditions (2) and (3) above, one can easily show that if  $L$ is a trivial link with $n$ connected components, then $P(L,z,a)=((a-a^{-1})z^{-1})^{n-1}$ (by applying these two conditions repeatedly to a simple closed curve with $n-1$ twists in its projection).
For our purposes, we will actually be using the following two equivalent forms of the skein relation:
\begin{eqnarray}
P(L_+,z,a)&=&a^{-2}P(L_-,z,a)+a^{-1}zP(L_0,z,a),\label{Skein1}\\
P(L_-,z,a)&=&a^2 P(L_+,z,a)-azP(L_0,z,a).\label{Skein2}
\end{eqnarray}

A rooted and edge-weighted binary tree $\T$ is called a {\em resolving tree} of an oriented link diagram $L$ (for the HOMFLY polynomial)
if the following conditions hold. First, every vertex of $\T$ corresponds to an oriented link diagram. Second, the root vertex of $\T$ corresponds to the original link diagram $L$. Third, each leaf vertex of $\T$ corresponds to a trivial link. Fourth, if we direct $\T$ using the directions of the paths from the root vertex to the leaf vertices, then under this direction any internal vertex has exactly two children vertices and the corresponding link diagrams of these three vertices are identical except at one crossing and they are related by one of the two possible relations at that crossing as shown in Figure~\ref{fig:rtree}, where the edges are weighted and the directions of the edges coincide with the direction of $\T$.

\begin{remark}\label{tree_formula_remark} {\em If $L$ admits a resolving tree $\T$, then one can easily show that $P(L,z,a)$ is a summation in which each leaf vertex of $\T$ contributes exactly one term in the following way. Let $\U$ be the trivial link corresponding to a leaf vertex in $\T$ and let $Q$ be the unique path from the root ($L$) to the leaf vertex ($\U$). Then the contribution of the leaf vertex is simply $((a-a^{-1})z^{-1})^{\gamma(\U)-1}$ multiplied by the weights of the edges in $Q$, where $\gamma(\U)$ is the number of components in $\U$. It is known (and not hard to prove) that resolving trees exist for any given oriented link diagram $L$, and that they are not unique in general. If $L^\p$ is the mirror image of $L$, a resolving tree for $L^\p$ can be obtained from a resolving tree of $L$ by taking mirror images of all link diagrams in it and replacing $a$ by $a^{-1}$ and $z$ by $-z$ in the edge weights. It follows the well known fact that $P(L^\p,z,a)=P(L,-z,a^{-1})$.}
\end{remark}

\begin{figure}[htb!]
\includegraphics[scale=.47]{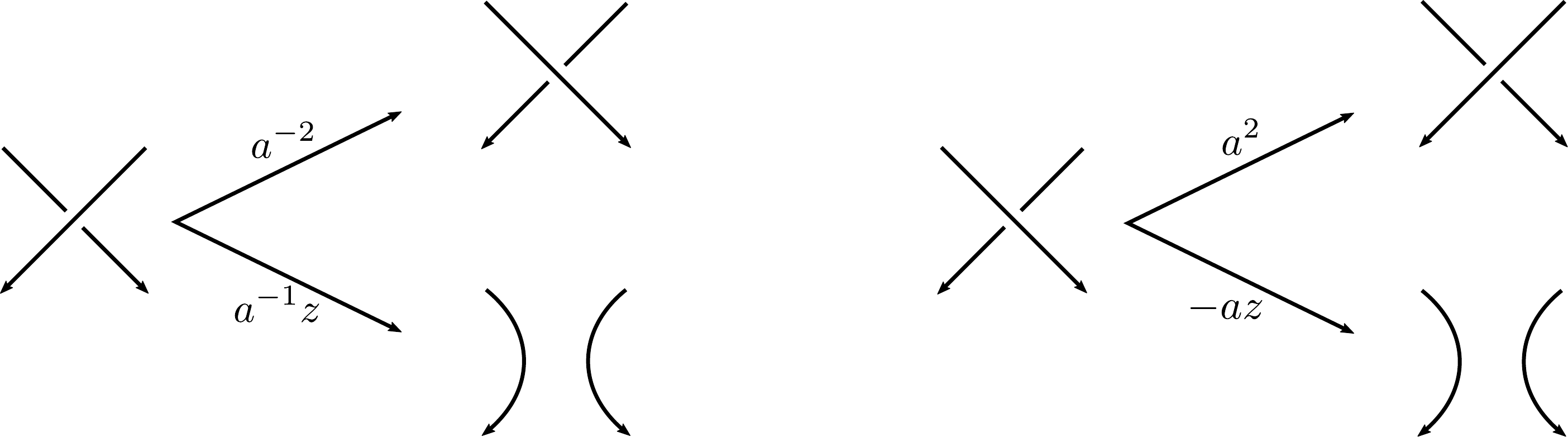}
\caption{A pictorial view of edge weight assignment by the skein relations (\ref{Skein1}) and (\ref{Skein2}). }
\label{fig:rtree}
\end{figure}

Let $\N$ be a braid on $n$ strands and let us define an algorithm (let
us call it Algorithm D) that operates on $\N$ as follows. If $\N$ is
descending, then the algorithm does not do anything and simply returns
$\N$. If $\N$ is not descending, then it contains at least one ascending
crossing. In this case, let us travel through $\N$ naturally, until we
run into the first ascending crossing $c$, which can be a positive or a
negative crossing as shown in either the left hand side or the right
hand side of Figure~\ref{fig:rtree}. The algorithm then returns the two
corresponding braids split from the original one as shown in
Figure~\ref{fig:cpd}, which we will name as $\N_f$ and $\N_s$ ($f$ for
``flipping" the crossing and $s$ for ``smoothing" the crossing). Notice
that $\N_f$ has one less ascending crossing than $\N$ does since $\N_f$
and $\N$ are identical (including the return order of the strands)
except at the crossing $c$. But one cannot say the same for $\N_s$ since
it does not share the same strands (and the return order of the strands)
with $\N$. However, $\N$ and $\N_s$ share the same strands and crossings
up to crossing $c$ (which are all descending) while travel through them
naturally, and $\N_s$ has one less crossing than $\N$. So if Algorithm D
is repeatedly applied to a braid and the resulting braids of this
operation, then this process will end after a finite number of repeats
(in fact this number is bounded above by the number of crossings in
$\D$). It follows that we can construct a special resolving tree $\T$
for ${\D}$ (as a link diagram) as follows. We apply Algorithm D to $\D$
first if it is not descending, and then apply the algorithm again to the
two resulting braids (if they are not descending), and so on, until all
the leaf vertices are descending braids. The closures of the braids
involved in the process are the vertices of $\T$. In particular, the
closures of the resulting descending braids form the leaf vertices of
$\T$ since they are all trivial links. By assigning appropriate weights
to the edges of $\T$, one can easily verify that $\T$ is indeed a
resolving tree of ${\D}$. By the way it is constructed, $\T$ is
unique. In a similar fashion, we can also construct another (unique)
resolving tree of ${\D}$ by replacing ``descending" with ``ascending" in
the above (the algorithm corresponding to Algorithm D would be called
Algorithm A). To distinguish the two, we will call the first the
``descending tree" and the later the ``ascending tree" of ${\D}$, and
denote them by $\T^{\downarrow}({\D})$ and $\T^{\uparrow}({\D})$
respectively. An example of a descending tree is shown in the right hand
side of Figure~\ref{fig:cpd}.

\begin{figure}[htb!]
\includegraphics[scale=.35]{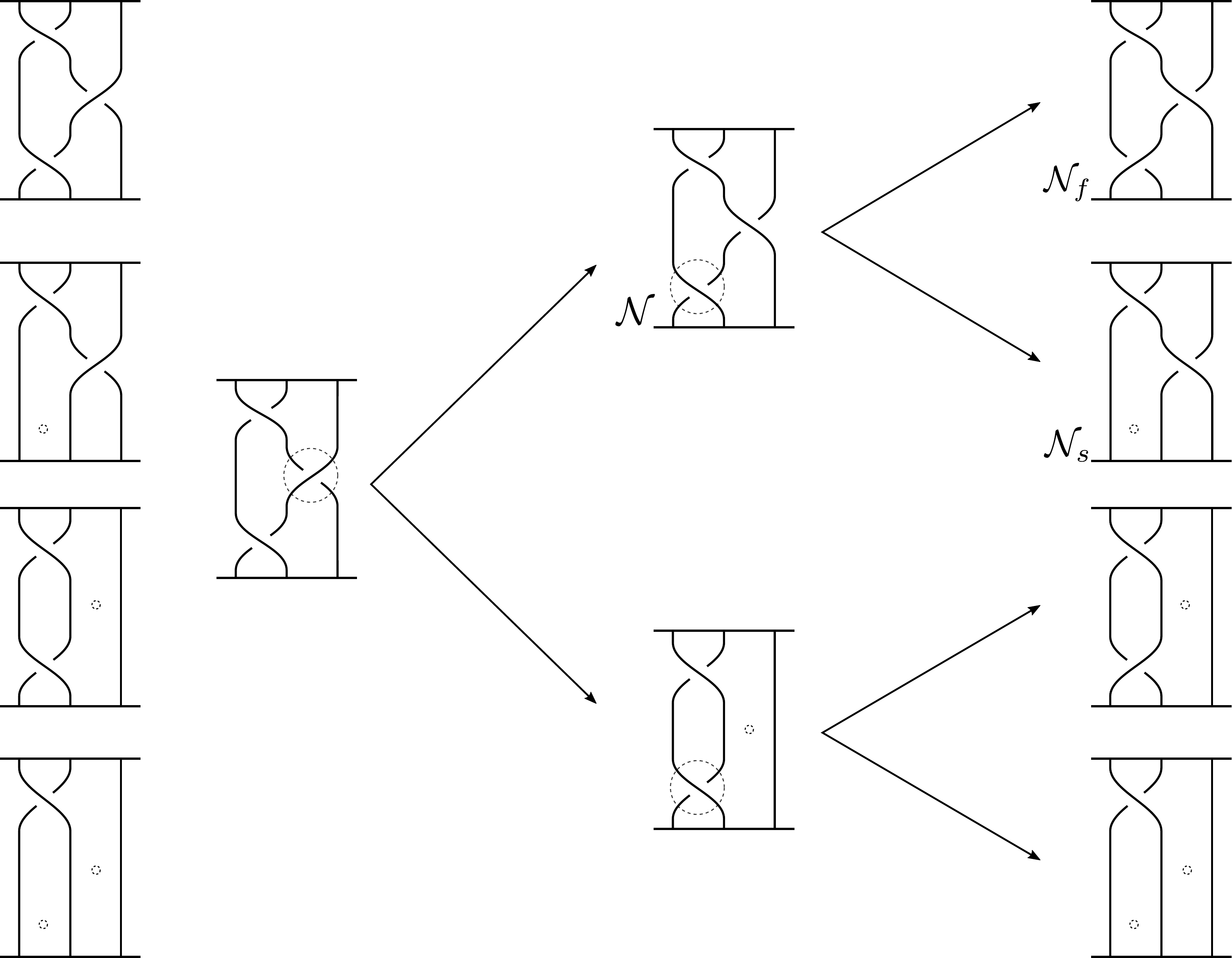}
\caption{Left: Admissible circuit partitions of the braid on the top (including itself), to be covered in Section \ref{acp}; Right: The descending tree of (the closure of) the braid on the top left (only the braids are shown). }
\label{fig:cpd}
\end{figure}

\begin{remark}\label{R_cor}
{\em For a given braid $\U$ obtained from $\D$ by flipping and smoothing
  some of its crossings, there is a simple way to check whether $\U\in
  \F^{\downarrow}(\D)$ by the way $\T^{\downarrow}(\D)$ is
  constructed. We travel through $\U$ naturally and visit each crossing
  of $\D$ exactly twice.  For each crossing 
of $\D$ that we encounter in this process for the first time (including the smoothed
ones, these are marked by small circles in Figure \ref{fig:cpd}) we
perform the following test. If this crossing is smoothed in $\U$,
we check whether we would be approaching it from its underpass if the crossing (in its original form in $\D$) were not smoothed.
On the other hand, if the crossing is also a crossing in $\U$ (which may or may not have been flipped from its original form in $\D$), we check whether we are approaching it from its overpass. If all crossings pass this
check, then $\U\in \F^{\downarrow}(\D)$, otherwise $\U\not\in
\F^{\downarrow}(\D)$.} 
\end{remark}

%\bigskip
\section{The HOMFLY polynomial of a closed braid}\label{s3}

In this section, we derive the main result of this paper, namely the two formulas of the HOMFLY polynomial of a  braid ${\D}$ based on $\T^{\downarrow}({\D})$ and $\T^{\uparrow}({\D})$ respectively, given in Theorem \ref{p3}. Let $\N$ be a vertex in $\T^{\downarrow}({\D})$ (or $\T^{\uparrow}({\D})$), $w(\N)$ be the writhe of $\N$, $\gamma(\N)$ be the number of components in ${\N}$. Note that $\N$ is obtained from $\D$ by applying Algorithm D (or Algorithm A) repeatedly, and in this process some crossings of $\D$ may have been smoothed. Let $t(\N)$ be the number of smoothed crossings and $t'(\N)$ be the number of smoothed crossings that are negative. It is trivial to note that $t(\N)$ is simply the difference between the number of crossings in $\D$ and the number of crossings in $\N$.

\begin{lemma}\label{l3}
If $\U$ is a descending braid on $n$ strands, then
$\gamma(\U)-w(\U)=n$. On the other hand, if $\V$ is an ascending braid on $n$ strands, then $\gamma(\V)+w(\V)=n$.
\end{lemma}

\begin{proof}
Since $\U$ is descending, by Remark~\ref{descending_remark}, ${\U}$ can
be realized by such space curves that its components are separated by
planes that are parallel to the $xy$-plane, hence the writhe
contribution of crossings whose strands belong to different components
is zero. Let $\tau=\gamma(\U)$, $\C_1$, $\C_2$, \ldots, $\C_{\tau}$ be the
cycles of $p(\U)$ and $l(\C_j)$ be the length of the cycle $C_j$. We
have $\sum_{1\le j\le \tau}l(C_j)=n$ and $\sum_{1\le j\le
  \tau}w(C_j)=w(\U)$. We claim that for each component $\C_j$ of ${\U}$,
$l(\C_j)=1-w(\C_j)$. For simplicity let $m=l(\C_j)$ and let the labels
of the strands in $\C_j$ be $s_1$, \ldots, $s_m$, ordered by their return
order in $\U$. Without loss of generality, assume that strand $s_{1}$ is
in the plane $z=m$, strand $s_2$ is in the plane $z=m-1$, ... and strand
$s_m$ is in the plane $z=1$. Furthermore, the curve connecting the
ending point of $s_i$ to the starting point of strand $s_{i+1}$ is
bounded between the two planes $z=m-i+1$ and $z=m-i$ for $1\le i\le
m$. We first observe that each strand and part of the curves connecting
to its ends can be deformed to straight line segments (within the plane
that it is in), in a form as shown in the left hand side of Figure
\ref{fig:layers}, where each strand resides in a plane parallel to $z=0$ (and the equation of the plane is marked in the figure). 
The part of the connecting curve attached to the strand residing in the same plane is marked by solid lines and the dotted curves have their end points on
different planes parallel to the $xy$-plane whose $z$ coordinates differ
by exactly one, with the exception of the dotted curve on the far left
(which is between the planes $z=1$ and $z=m$).  
\begin{figure}[htb!]
\includegraphics[scale=.6]{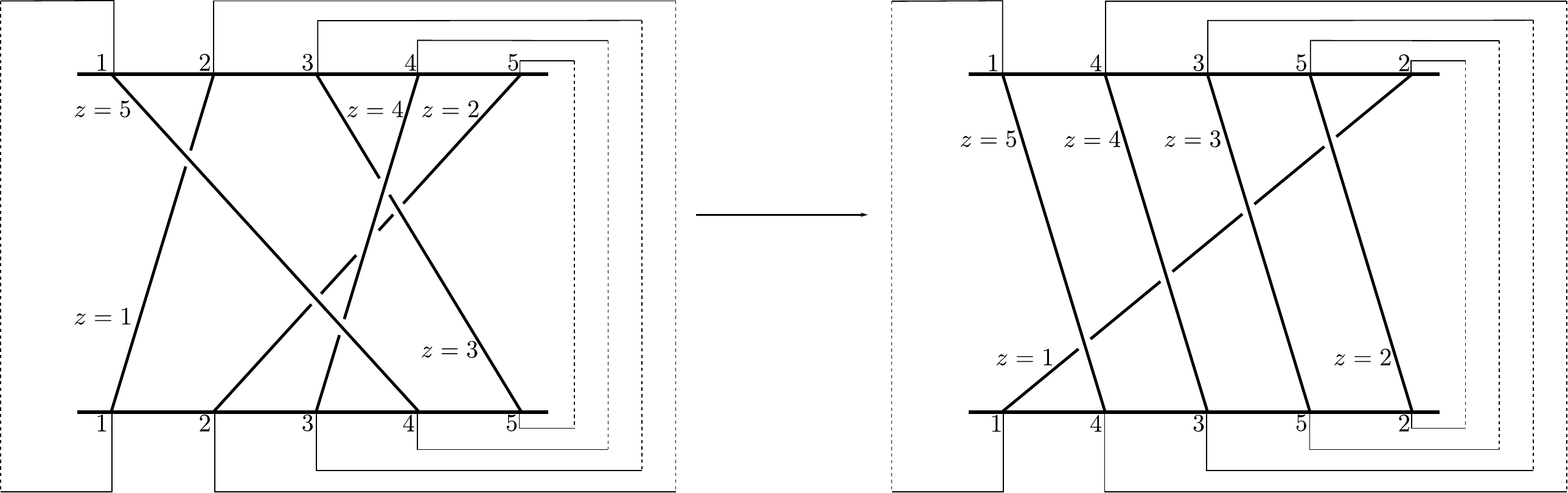}
\caption{ Left: A connected component corresponding to the cycle $\C_j=(14352)$ with its strands and the connecting curves straightened. Right: A special re-arrangement of the strands of $\C_j$ through a finite sequence of Reidemeister moves of type II and III.}
\label{fig:layers}
\end{figure}

By Remark \ref{R_remark},
this deformation does not change the writhe. For the same reason, these straight
line segments can freely slide within the plane they reside in since
there are no other curves in that plane. See Figure \ref{slide} for one
such move. In particular, nothing can preventing us from sliding them
into the position where the strands $s_1,\ldots,s_{m-1}$ are parallel
lines arranged in this order from left to right, as shown in the right
hand side of Figure~\ref{fig:layers}. Since all moves are made within
the planes where these curves reside in, the writhe does not change by
Remark \ref{R_remark}. We see that $w(\C_j)=-(m-1)$ from the right hand side of Figure~\ref{fig:layers} since there are exactly $m-1$ crossings in the projection and all of them are negative. Thus $w(\U)=\sum_{1\le j\le \tau}w(C_j)=-\sum_{1\le j\le \tau}(l(C_j)-1)=-n+\tau$, {\em i.e.}, $\gamma(\U)-w(\U)=n$. An ascending braid diagram $\V$ is the mirror image of a descending braid
diagram $\U$. It is known that $w(\U)=-w(\V)$ and it follows that
$\gamma(\V)+w(\V)=n$. 
\end{proof}

\begin{figure}[htb!]
\includegraphics[scale=.6]{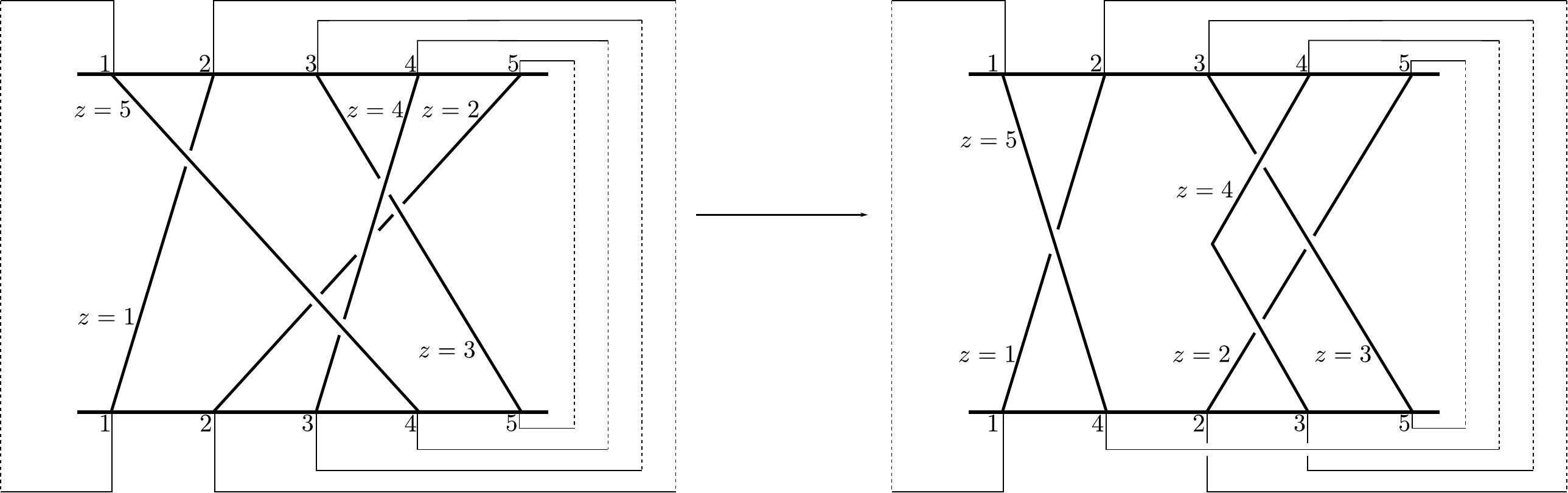}
\caption{Left: The strand with label $1$ is the only strand in the top plane and can be slid freely within the plane without causing self intersection of the link; Right: The same strand after a sliding move.}
\label{slide}
\end{figure}

\medskip
\begin{theorem}\label{p3}
Let $\D$ be a braid on $n$ strands, $\F^{\downarrow}({\D})$ and $\F^{\uparrow}({\D})$ be the set of leaf vertices of $\T^{\downarrow}({\D})$ and $\T^{\uparrow}({\D})$ respectively, then 
\begin{equation}\label{e1}
P({\D},z,a) = a^{1-n-w({\D})} \sum_{\U\in \F^{\downarrow}({\D})} (-1)^{t'(\U)}z^{t(\U)}((a^2-1)z^{-1})^{\gamma(\U)-1}
\end{equation}
\begin{equation}\label{e2}
P({\D},z,a) = a^{n-1-w({\D})} \sum_{\V\in \F^{\uparrow}({\D})} (-1)^{t'(\V)}z^{t(\V)}((1-a^{-2})z^{-1})^{\gamma(\V)-1}
\end{equation}
where $t(\U)$ is the number of crossings in $\D$ that are smoothed in obtaining $\U$ and $t^\p(\U)$ is the number of negative crossings among these smoothed crossings.
\end{theorem}

\begin{proof} Let us consider the descending tree first. By Remark~\ref{tree_formula_remark}, the contribution of $\U\in \F^{\downarrow}({\D})$ to
$P({\D},z,a)$ is $((a-a^{-1})z^{-1})^{\gamma(\U)-1}$ multiplied by the
  weights of the edges on the unique path of $\T^{\downarrow}({\D})$
  from ${\D}$ to ${\U}$. As shown in Figure~\ref{fig:rtree}, the degree
  of $a$ in the weight of an edge is exactly the change of writhe from
  the starting vertex of the edge (remember that it is directed from the
  root to the leaf) to the ending vertex of the edge, whereas a $z$ term
  in the weight of the edge indicates that the ending vertex is obtained
  from the starting vertex by a crossing smoothing and a negative sign in the weight indicates that the smoothed crossing is a negative crossing. It follows that the total contribution of 
$\U$ is
\begin{eqnarray*}
 &&(-1)^{t'(\U)}z^{t(\U)}a^{w(\U)-w({\D})}((a-a^{-1})z^{-1})^{\gamma(\U)-1}\\
 &=& (-1)^{t'(\U)}z^{t(\U)}a^{w(\U)-w({\D})-\gamma(\U)+1}((a^2-1)z^{-1})^{\gamma(\U)-1}\\
 &=& a^{1-n-w({\D})}\left[(-1)^{t'(\U)}z^{t(\U)}((a^2-1)z^{-1})^{\gamma(\U)-1}\right]
\end{eqnarray*}
by Lemma~\ref{l3}. This proves (\ref{e1}). Equation (\ref{e2}) can be proved in a similar fashion and is left to the reader.
\end{proof}

\begin{remark}{\em 
Let $L$ be a link with braid index $n$ and  ${\D}$ a   braid representation of $L$ where $\D$ is a braid of $n$ strands. Let $E$ and $e$ be the maximum and minimum degrees of $a$ in $P(L,z,a)$. The $a$-span of $P(L,z,a)$ is defined as $E-e$. Since $\gamma(\U)\le n$ for any $\U\in \F^{\downarrow}({\D})$,
formulas (\ref{e1}) and (\ref{e2}) imply that $E\le 1-n-w({\D})+2(n-1)=n-w({\D})-1$ and $e\ge n-1-w({\D})-2(n-1)=-n-w({\D})+1$. It follows that $a$-span$/2+1\le n$, that is, the Morton-Frank-Williams inequality is a direct consequence of Theorem \ref{p3}.}
\end{remark}
 
\section{The braid index of reduced alternating braids}\label{s4}

A link is {\em splittable} if there exist components of the link that lie on different sides of a topological plane in $\R^3$ and a link is {\em non-splittable} if it is not splittable~\cite{A}. A braid diagram is {\em reduced} if its closure is a reduced link diagram. A braid diagram is {\em alternating} if its closure is an alternating link diagram. An oriented link is a {\em reduced alternating braid} on $n$ strands if it can be represented by a reduced alternating braid diagram on $n$ strands. A reduced alternating braid is an alternating link but the converse is not true in general. Figure \ref{fig:bd} is an example a non-splittable, reduced alternating braid diagram. In this section,  we prove the following theorem with a simple and direct proof based on our results from the last section.

\begin{theorem}\cite{Mu}\label{Theorem4.1} Let $\D$ be the closure of a reduced alternating braid on $n$ strands, then the braid index of $\D$ is $n$.
\end{theorem}

\begin{remark}{\em
The above theorem is a special case of a more general theorem on a class of oriented alternating fibered links due to Murasugi \cite{Mu} where the proof relies on using a sequence of lemmas shown by induction. The alternating links in this class are the $^\ast$-products of $(2,n)$ torus links (which include the alternating closed braids) \cite{Mu} and are also known to be fibered \cite{Sto2}. The fact that the reduced alternating closed braids are fibered can be established directly by using another result due to Murasugi, which states that an alternating link is fibered if the leading coefficient of its Alexander polynomial is $\pm 1$ \cite{Mu2}, and by proving that the leading coefficient of the Alexander polynomial of a reduced alternating closed braid is indeed $\pm 1$. This latter fact seems to be known but we failed to find a direct proof of it in the literature so we will provide a short one at the end of this section, which also serves as another application of  Theorem \ref{p3} and our method. We also note that fact that the reduced alternating closed braids are fibered is a direct consequence of a result due to Stallings in which he proved that all homogeneous closed braids (which include the alternating ones) are fibered \cite{Stallings}.}
\end{remark}   

Before we proceed to the proof of the theorem, we note that it suffices for us to prove this result for  reduced alternating braids that are non-splittable. Since if not, say $\D=\D_1\cup \D_2\cup \cdots\cup \D_k$ with the $\D_j$'s being the non-splittable components of $\D$, we can simply apply our result to each $\D_j$. Since the braid index of a link equals the sum of the braid indices of its non-splittable components, and $a$-span$/2+1$ of $P(\D,z,a)$ is the sum of  the ($a$-span$/2+1$)'s of the $P(\D_j,z,a)$'s as one can easily check, the general result then follows. 

\begin{figure}[htb!]
\includegraphics[scale=.35]{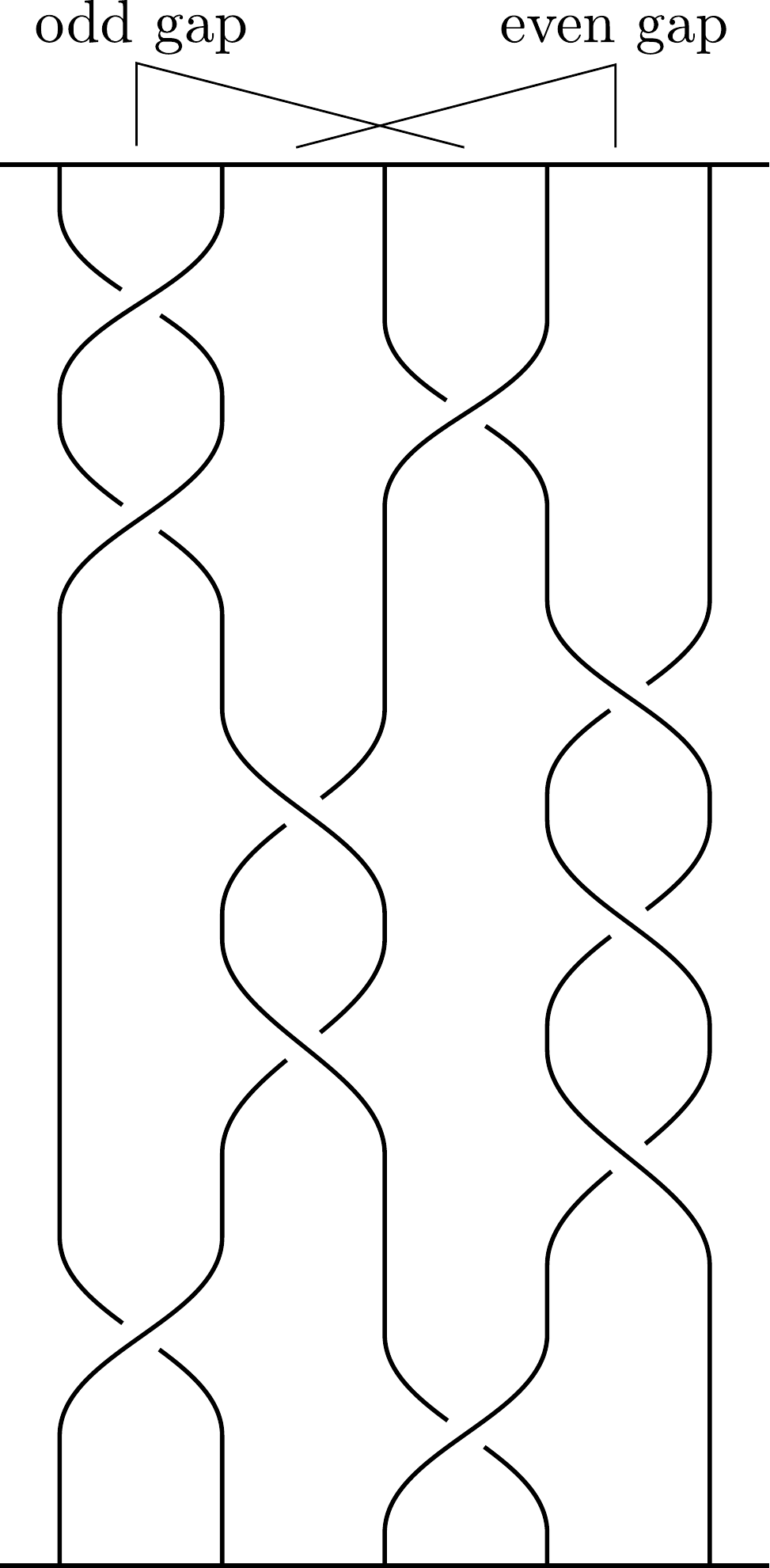}
\caption{A non-splittable reduced positive-leading alternating braid}
\label{fig:bd}
\end{figure}

Let $\D$ be a braid diagram on $n>1$ strands. Consider $\D$ as a word of generators of the braid group $B_n$. Note that each crossing in the braid diagram corresponds to a letter $\sigma_i$ or $\sigma_i^{-1}$ in the word $\D$. In a standard drawing of $\D$ as a braid, a crossing corresponding to $\sigma_i$ or $\sigma_i^{-1}$ is drawn in the space between the vertical line $x=i$ and $x=i+1$ in the $xy$-plane. We call this space a {\em gap}. A gap is {\em odd} if $i$ is odd, and {\em even} if $i$ is even. Thus, a crossing corresponds to $\sigma_i$ or $\sigma_i^{-1}$ is in an {odd gap} ({even gap}) if $i$ is odd (even). If $\D$ is non-splittable, then each gap in the braid diagram contains at least one crossing. If $\D$ is also reduced, then each gap in the diagram contains at least two crossings. Suppose $\D$ is alternating, the overpass (underpass) of a crossing $c$ must be the underpass (overpass) at next crossing $c^\prime$. So if $c^\prime$ is in the same gap, $c^\prime$ must have the same sign as $c$ and if $c^\prime$ is in an adjacent gap, $c^\prime$ must have the opposite sign. Therefore, if $\D$ is alternating, all odd gaps contain crossings of the same sign and all even gaps contain crossings of the opposite sign. If the first gap contains positive crossings, we say the alternating braid diagram $\D$ is {\em positive-leading} and if the first gap contains negative crossings, we say the alternating braid diagram $\D$ is {\em negative-leading}.

\begin{lemma}\label{l4}
Let $\D$ be a non-splittable reduced alternating braid diagram on $n$ strands, $E$ be the highest degree of $a$ in $P(\D,z,a)$ and $e$ the lowest degree of $a$
in $P(\D,z,a)$. Then $E=n-1-w(\D)$ and $e=1-n-w(\D)$. It follows that $a$-span$=2(n-1)$. 
\end{lemma}

A simple application of the Morton-Frank-Williams inequality and Lemma~\ref{l4}, together with the note in the second paragraph of this section, immediately shows that if $\D$ is a reduced alternating braid on $n$ strands, then the braid index of $\D$ is exactly $n$. We now proceed to prove Lemma \ref{l4}. 

\begin{proof} We will consider the positive-leading alternating braid diagrams first. Bear in mind that for a positive-leading alternating braid diagram, all the odd gaps contain only positive crossings and all the even gaps contain only negative crossings.

Consider $P(\D,z,a)$ as a Laurent
polynomial of $a$ with coefficients in the ring of Laurent polynomials
in the variable $z$. By (\ref{e1}), the highest possible degree of $a$ is $n-1-w(\D)$ and the only leaf vertices $\U\in \F^{\downarrow}(\D)$ that can make contributions to the 
term of $P(\D,z,a)$ with this $a$ degree must satisfy the condition $\gamma(\U)=n$. Let us consider
the braid $\U^*$ obtained from $\D$ by first smoothing all crossings in
the odd gaps, then smoothing all crossings in even gaps except the first one and the last one and finally
flip the sign of the last crossings in the even gaps.

Claim 1. $\U^*\in \F^{\downarrow}(\D)$. This is obvious by the checking method in Remark \ref{R_cor}. 

Claim 2. Part of the $\U^*$ contribution to $P(\D, z, a)$ is a term of
the form $\pm z^{t(\U^*)-n+1} a^{n-1-w(\D)}$. 

Proof of Claim 2. By (\ref{e1}), the contribution of $\U^*$ to $P(\D, z, a)$ is
$$
a^{1-n-w({\D})}(-1)^{t'(\U^*)}z^{t(\U^*)}((a^2-1)z^{-1})^{\gamma(\U^*)-1}.
$$
We have $\gamma(\U^*)=n$ and the result follows.

Claim 3. For any $\U\in \F^{\downarrow}(\D)$, $\U\not= \U^*$, the contribution of $\U$ to 
 $P(\D, z, a)$ either has its  maximum degree in the variable $a$ less than $n-1-w(\D)$, or has its degree in $z$ less than $t(\U^*)-n+1$.

Proof of Claim 3. If $\gamma(\U)<n$ then there is nothing to prove. Assuming that $\gamma(\U)=n$, then the contribution of $\U$ to $P(\D,z,a)$ is $a^{1-n-w({\D})}(-1)^{t'(\U)}z^{t(\U)}((a^2-1)z^{-1})^{n-1}$ so the degree of $z$ is $t(\U)-n+1$. We need to show that $t(\U)-n+1<t(\U^*)-n+1$, that is, $t(\U)<t(\U^*)$. Since $\gamma(\U)=n$, the return order of $\U$ is $1\tl 2\tl\cdots\tl n$. It follows that each gap contains either no crossings or at least two crossings of $\U$ since each strand, say it has label $i$ with starting point $(i,1)$, has to go through a gap an even number of times in order for it to end at the point $(i,0)$. We claim that $\U$ contains the first crossing of $\D$ in each even gap in its original sign. If this is not the case, let $c$ be the first such crossing that has been changed (so it is either smoothed or flipped in $\U$). Assume that $c$ is in the $i$-th gap (with $i$ being even). Notice that any strand with label $j$ less than $i$ can only cross the $i$-th gap into gaps to the right of the $i$-th gap and must return to the point $(j,0)$ at the end because the return order of $\U$ is $1\tl 2\tl\cdots\tl n$ and the strands can only travel downward. It follows that the strand entering $c$ from its right side must have label greater than $i$. Now consider the strand of $\U$ that first enters $c$ as we travel through $\U$ naturally. This strand has to come from the left side of $c$ by the above observation. But then it fails the test given in Remark \ref{R_cor} since $c$ is descending in $\D$ and cannot be changed in $\U$. This gives us the needed contradiction.
It now follows that $\U$ has at least two crossings in each even gap. If $\U$ also contains some crossings in the odd gaps, or contains more than two crossings in some even gaps, then we already have $t(\U)<t(\U^*)$ and there is nothing left to prove. So the only case left is the case when $\U$ contains no crossings in the odd gaps and exactly two crossings in each even gap. 

Claim 4. If $\U\in \F^{\downarrow}(\D)$ and it contains no crossings in the odd gaps and exactly two crossings in each even gap, then $\U=\U^*$, that is, $\U^*$ is the only element in $\F^{\downarrow}(\D)$ with this property.

Proof of Claim 4. By the proof of Claim 3, if $\U\not=\U^*$, then there exists an even gap such that $\U$ contains the first crossing of $\D$ in this gap in its original sign, and exactly one other crossing $c$ of $\D$ in this gap which is not the last crossing in this gap. 
By Remark \ref{R_cor} again, the sign of $c$ in $\D$ has to be changed to make it descending in $\U$. But as we travel through $\U$ naturally past $c$ and encounter the first crossing of $\D$ in the same gap below $c$ (keep in mind that this crossing exists because $c$ is not the last crossing of $\D$ in this gap, and the other crossings of $\D$ in the adjacent odd gap have all been smoothed in $\U$). This crossing has been smoothed in $\U$ but we are now approaching it from its overpass, so it fails the check in Remark \ref{R_cor} and $\U\not\in \F^{\downarrow}(\D)$, contradicting to $\U\in  \F^{\downarrow}(\D)$.

The consequence of Claims 1 to 4 is that if we write $P(\D,z,a)$ as a Laurent polynomial of $a$ whose coefficients are Laurent polynomials of $z$, then it contains a nontrivial term of the form $g(z)a^{n-1-w(\D)}$ and all other terms have degrees less than $n-1-w(\D)$, that is, $E=n-1-w(\D)$. To obtain $e$, we will use $\V^*\in \T^{\uparrow}(\D)$ and (\ref{e2}), where $\V^*$ is obtained from $\D$ by keeping the first crossing and flipping the last crossing in each odd gap, and smoothing all other crossings. The details are left to the reader.

Finally, if $\D$ is a non-splittable reduced negative-leading alternating braid diagram, then its mirror image $\D^\p$ is a non-splittable reduced positive-leading alternating braid diagram and we have $w(\D)=-w(\D^\p)$. Let $E^\p$ and $e^\p$ be the highest and lowest degrees of $a$ in $P(\D^\p,z,a)$. Then by the first part of the lemma, we have $E^\p=n-1-w(\D^\p)$ and $e^\p=1-n-w(\D^\p)$. Therefore by Remark~\ref{tree_formula_remark} we have $E=-e^\p=-(1-n-w(\D^\p))=n-1-w(\D)$ and $e=-E^\p=-(n-1-w(\D^\p))=1-n-w(\D)$.   
\end{proof}
\begin{remark}\label{remark4.4}
{\em Another way to handle the case when $\D$ is a non-splittable
  reduced negative-leading alternating braid diagram is to take its
  connected sum with a positive Hopf link, thus creating a new first gap
with a single pair of positive crossings. The statement now follows the
positive leading case and from the fact that the braid index is additive
under taking the connected sum.}
\end{remark}

Finally, as another application of Theorem \ref{p3}, we provide a short proof to the following theorem. This result seems to be known \cite{Stallings}, although we failed to find a specific proof in the literature.

\begin{theorem}\label{T4.4}
Let $\D$ be the closure of a reduced alternating braid on $n$ strands, then the leading coefficient in the Alexander polynomial of $\D$ is $\pm 1$.
\end{theorem}

\begin{proof}
Let $\Delta_\D(x)$ be the Alexander polynomial of $\D$, then $\Delta_\D(x)=P(\D,x^{1/2}-x^{-1/2},1)$ \cite{Doll}. Assume that $\D$ is positive leading, substituting $a=1$, $z=x^{1/2}-x^{-1/2}$ in (\ref{e1}) leads to
$$
\Delta_\D(x)=  \sum_{\U\in \F^{\downarrow}({\D}), \gamma(\U)=1} (-1)^{t'(\U)}(x^{1/2}-x^{-1/2})^{t(\U)}.
$$
Construct $\U^\p\in \F^{\downarrow}({\D})$ by the following procedure: (i) for the first gap of $\D$, smooth all crossings (which are positive) except the last one, which we will flip and cross it into the second gap. Once there, we have no choice but to keep the first (negative) crossing we encounter which will lead us into the third gap. (ii) Now we will smooth all crossings we encounter (negative and positive ones in the second and the third gaps) until we encounter the last positive crossing in the third gap. Notice that in doing so we may have reached the bottom of the braid at $\{3\}\times \{1\}$ and returned to the top of the braid at $\{3\}\times \{0\}$, so this last positive crossing we encounter may not be the last positive crossing in the third gap. At this positive crossing we flip it and cross it into the next gap and repeat (ii). This process is repeated until we reach the last gap and we can and will smooth all crossings in the last gap except one, no matter it is even or odd. At this point, we have exactly one positive crossing in each odd gap but may have negative crossings in the even gaps that we have not visited. However in the last gap we only have one crossing left and we will be traveling back through this crossing. As one can easily verify, each time we travel back to an even gap it is through the only crossing left in the odd gap to its right, so if there are any negative crossings that we have not visited before, these crossings will be between the two positive crossings in the two odd gaps to the left and right of the said even gap. By the descending rule, we can and will smooth all these crossings. Thus we have shown that 
$\U^\p\in \F^{\downarrow}({\D})$, $\gamma(\U^\p)=1$ and $t(\U^\p)=c(\D)-n+1$ where $c(\D)$ is the total number of crossings in $\D$. Furthermore, it is also rather easy to see that $t(\U)>c(\D)-n+1$ implies $\gamma(\U)>1$. Finally we leave it to our reader to verify that $\U^\p$ is the only element in $\F^{\downarrow}({\D})$ with the property $t(\U)=c(\D)-n+1$. The result of the theorem then follows. If $\D$ is negative leading, we will simply apply the above argument to (\ref{e2}) instead.
\end{proof}

\section{Admissible circuit partitions}\label{acp}

In this section we show that the HOMFLY polynomial formation (\ref{e1}) given in Theorem~\ref{p3}
is equivalent to the expansion derived by F.\ Jaeger for   braids in \cite{Ja}. Thus our approach used to prove  Proposition~\ref{p3}
provides an alternative (and in fact shorter) proof of his expansion. Unlike our approach (which is combinatorial in nature), 
Jaeger stated his formula using a concept called the {\em admissible circuit partitions} of a braid diagram. We establish 
this equivalence by showing that for any braid $\D$, there is a bijection between the leaf
vertices in the descending resolving tree $\T^{\downarrow}({\D})$ and the admissible circuit partitions of $\D$, such that
each leaf vertex and its corresponding admissible partition under this bijection make the same contributions to their respective 
HOMFLY polynomial formulas.

We begin with a brief review of the essential concepts introduced by
F.\ Jaeger with some small modification of the terminology to make the terms more consistent with our current paper.  
Interested readers please refer to \cite{Ja} for the details and his original terminologies. 

Given a braid diagram $\D$, a circuit partition $\pi$ of $\D$ is a braid diagram obtained 
from $\D$ by smoothing every crossing in a
subset $S$ of its crossings, while leaving the other crossings
unchanged. The smoothed crossings are denoted by small dotted circles in Figure 5. We can identify a circuit partition $\pi$ by the ordered pair 
$(\D,S)$ where $\D$ is the original braid diagram and
$S$ the set of crossings of $\D$ to be smoothed. Two circuit partitions $(\D,S)$ and $(\D',S')$ 
are defined to be equal if and only if $\D=\D'$ and $S=S'$. 

Let $c$ be a crossing in $\D$.
If $c$ is smoothed, the upper left part of the strand entering $c$ is connected to the 
lower left part of the strand exiting $c$ and the resulting curve is called a {\em left tangence} at $c$. 
Similarly one can define the {\em right tangence} at $c$. 

While traveling through $\pi$ naturally, we will meet each crossing in $\D$ exactly
twice (including the crossings that have been smoothed). A smoothed crossing $c$ in $\pi$ is said to be {\em admissible}
if it has the following property.  If $c$ is a positive crossing,
then the first passage at $c$ is a left tangence, if $c$ is negative
in $\D$, then the first passage at $c$ is a right tangence.
A circuit partition is {\em admissible} if all the smoothed crossings in $\pi$ are admissible.
In particular, $\D$ itself is an admissible circuit partition, in which no crossing is smoothed.
We denote the set of admissible circuit partitions of $\D$ by $\A(\D)$.  Then in \cite{Ja} Jaeger showed that
for any braid diagram $\D$ on $n$ strands, 
\begin{equation}
P({\D},z,a)=a^{1-n-w({\D})} \sum_{\pi\in \A(\D)} (-1)^{t^\p(\pi)}z^{t(\pi)}((a^2-1)z^{-1})^{\gamma(\pi)-1}
\label{J_equation}
\end{equation}
where $\pi=(\D,S)$, $t(\pi)$ is the number of crossings in $S$, $t^\p(\pi)$ is the number of negative crossings in $S$ and $\gamma(\pi)$ 
is the number of components in the closure of $\pi$.

\begin{proposition}
\label{thm:da}  
 For any braid diagram $\D$, there exists a bijection $h_\D:\ \F^{\downarrow}({\D})\to \A(\D)$ 
 such that for each $\U\in \F^{\downarrow}({\D})$, $\U$ and $\pi=h_\D(\U)$ are both obtained 
 from $\D$ by smoothing crossings from the same set.  
 \end{proposition}  

\begin{proof}
Define a mapping $h_\D$ from $\F^{\downarrow}({\D})$ to the set of all circuit partitions of $\D$ as follows.
For any $\U\in  \F^{\downarrow}({\D})$, let $h_\D(\U)$ be 
the circuit partition $(\D,S)$ where $S$ is the set of crossings in $\D$ that are smoothed in 
the process of obtaining $\U$. We claim that $h_\D$ is a bijection between $\F^{\downarrow}({\D})$ and $\A(\D)$.

We proceed by induction on $k$, the number of vertices in the
descending tree $\T^{\downarrow}({\D})$. If $k=1$, then the tree consists of only the root vertex,
that is, $\D$ is descending. So $h_\D({D})=(\D,\emptyset)=\D\in \A(\D)$.
Assume that there is an admissible circuit partition $\pi\not=\D$, and let
$c$ be the first crossing encountered as we travel through $\pi$ naturally. Up to the crossing $c$,
 traveling through $\pi$ naturally is the same as  traveling through $\D$ naturally
since nothing has been changed up to that point. Since $c$ is descending, the first strand
entering $c$ is the top strand at $c$,  and smoothing $c$
results in a right tangence if $c$ is positive and in a left tangence if
$c$ is negative. This contradicts the definition of an admissible circuit partition. Thus the only
admissible circuit partition of $\D$ is itself. So $h_\D$ is a bijection.

Assume that the statement is true for all $\D$ whose descending tree has at most $k(\ge 1)$ vertices
 and let $\D$ be such that $\T^{\downarrow}({\D})$ has at most $k+1$ vertices.
Since $k+1\ge 2$, $\D$ contains at least one ascending crossing. Let
$c$ be the first ascending crossing of $\D$ encountered when we travel through $\D$ naturally. All crossings preceding $c$ in the
return order being descending, they are not switched or smoothed in any
vertex (braid) of $\T^{\downarrow}({\D})$. Let $\D_f$ and $\D_s$ be the children 
of the root vertex of $\T^{\downarrow}({\D})$, which are obtained by flipping 
and smoothing $c$ respectively. The descendants
of $\D_f$ and $\D_s$ respectively form the rooted trees 
$\T^{\downarrow}(\D_f)$ and $\T^{\downarrow}(\D_s)$ respectively. The set of leaves
$\F^{\downarrow}({\D})$ is the union of the sets $\F^{\downarrow}(\D_f)$ and
$\F^{\downarrow}(\D_s)$. Note that this union is disjoint, since $c$ is 
not smoothed in elements of $\F^{\downarrow}(\D_f)$ but is smoothed
in elements of $\F^{\downarrow}(\D_s)$. Using an argument similar to the one used in
the case $k=1$, we see that for every admissible circuit partition
$\pi=(\D,S)\in A(\D)$, $S$ does not contain any crossing preceding $c$ when we travel along the strands of $\D$ naturally Thus each $(\D,S)\in \A(\D)$ can be identified 
with $(\D_f,S)\in \A(\D_f)$ if $c\not\in S$, and with $(\D_s,S\setminus \{c\})\in \A(\D_s)$ if $c\in S$. 
Let $\jmath: \A(\D_f)\cup \A(\D_s)\to \A(\D)$ be the inverse of this identifying map (which is a bijection, of course).
Since $\D_f$ and $\D_s$ are the children 
of the root vertex of $\T^{\downarrow}({\D})$, $\T^{\downarrow}({\D_f})$ and $\T^{\downarrow}({\D_s})$
both have at least one less vertex than $\T^{\downarrow}({\D})$ does. By the induction hypothesis, $h_{\D_f}$ and $h_{\D_s}$
are bijections that map $\F^{\downarrow}({\D_f})$ to $\A(\D_f)$ and $\F^{\downarrow}({\D_s})$ to $\A(\D_s)$ respectively.
It follows that the mapping $h_{D_f\cup \D_s}: \F^{\downarrow}({\D_f})\cup \F^{\downarrow}({\D_s})\to \A(\D_f)\cup \A(\D_s)$ defined 
by $h_{\D_f\cup \D_s}(\pi)=h_{D_f}(\pi)$ if $\pi\in \F^{\downarrow}({\D_f})$ and $h_{\D_f\cup \D_s}(\pi)=h_{D_s}(\pi)$ if $\pi\in \F^{\downarrow}({\D_s})$
is a bijection. Since $h_\D=\jmath\circ h_{D_f\cup \D_s}$, $h_\D$ is a bijection between $\F^{\downarrow}({\D})$ and $ \A(\D)$.
This concludes our proof.
\end{proof}

\begin{example}
{\em
For the descending tree represented in Figure~\ref{fig:cpd}, the  
admissible circuit partitions corresponding to the leaf vertices are
depicted on the left hand side of the picture. For each leaf the
corresponding admissible circuit partition is at the same level. Note
that for each corresponding pair the same crossings of the braid at the root
are smoothed.} 
\end{example} 
 
The bijection introduced in Proposition~\ref{thm:da} has the property that 
$\U\in \F^{\downarrow}({\D})$ and
$h_{\D}(\U)$ are identical except the signs at some crossings. Hence we have
$
\gamma(\U)=\gamma(h_{\D}(\U))$,  $t(\U)=t(h_{\D}(\U))$ and 
$t^\p(\U)=t^\p(h_{\D}(\U))$. This leads to the following result.

\begin{corollary}
F.\ Jaeger's expansion of the HOMFLY polynomial (\ref{J_equation})~\cite[Proposition 3]{Ja} is an immediate consequence of 
 (\ref{e1}).
\end{corollary}

As a final note to this section, we point out that if we change the definition of admissible circuit partition by defining a smoothed crossing $c$ in $\pi$ to be { admissible}
if  the first passage at $c$ is a left tangence when $c$ is negative, and the first passage at $c$ is a right tangence if $c$ is positive, then we can use (\ref{e2}) to show that F.\ Jaeger's expansion of the HOMFLY polynomial (\ref{J_equation}) becomes 
\begin{equation}
P({\D},z,a)=a^{n-1-w({\D})} \sum_{\pi\in \A^*(\D)} (-1)^{t^\p(\pi)}z^{t(\pi)}((1-a^{-2})z^{-1})^{\gamma(\pi)-1}
\label{J_equation2}
\end{equation}
where $A^*(\D)$ is the set of admissible circuit partitions under this new definition.

\section{Ending remarks}

We end this paper by noting the potential application of Theorem \ref{p3} to other classes of links that are presented in a closed braid form. It is also an interesting yet challenging question to explore whether there exist other classes of links that allow formulations similar to the ones in Theorem \ref{p3}. These will be the future research directions of the authors.

%\smallskip
\section*{Acknowledgement}
The research of the third author was partially supported by a grant from
the Simons Foundation (\#245153 to G\'abor Hetyei). The authors thank
Alex Stoimenow for his insightful comments. They are also indebted to an
anonymous referee for the careful reading of this manuscript and
for suggesting several substantial improvements. The alternative argument given in Remark \ref{remark4.4} is also due to the referee.

\end{document}